\documentclass[english,11pt,reqno]{smfart}

\usepackage{amssymb}
\usepackage{amsmath}
\usepackage{amsthm}

\newcommand{\N}{\mathbb{N}}
\newcommand{\R}{\mathbb{R}}

\newcommand{\T}{\mathbb{T}}

\newtheorem{theorem}{Theorem}
\newtheorem{proposition}[theorem]{Proposition}
\newtheorem{lemma}[theorem]{Lemma}

\newtheorem{definition}[theorem]{Definition}
\newtheorem{remark}{Remark}

\def\di{\displaystyle}

\def\di{\displaystyle}
\def\N{\mathbb{N}}
\def\R{\mathbb{R}}
\def\T{\mathbb{T}}
	\def\TK{\mathbb{T}^\kappa}
	\def\Tk{\mathbb{T}_\kappa}
	\def\TKk{\mathbb{T}^\kappa_\kappa}

\def\L{\mathcal{L}}

\def\DD{\Delta}
\def\NN{\nabla}

\def\Crd{C^0_{\mathrm{rd}}}
\def\Cdrd{C^{1,\DD}_{\mathrm{rd}}}
\def\Cdrdz{C^{1,\DD}_{\mathrm{rd},0}}

\def\Cld{C^0_{\mathrm{ld}}}
\def\Cnld{C^{1,\NN}_{\mathrm{ld}}}

\def\RS{\mathrm{RS}}
\def\LS{\mathrm{LS}}
\def\RD{\mathrm{RD}}
\def\LD{\mathrm{LD}}

\def\ss{\mathrm{S}}

\def\CDN{\Cdrd(\T)\times\Cnld(\T)}
\def\CDNI{C^{1,\DD\times\nabla}}
\def\CDNz{C^{1,\DD\times\nabla}_0}

\newcommand{\fonction}[5]{\begin{array}[t]{lrcl}#1 :&#2 &\longrightarrow &#3\\&#4& \longmapsto &#5 \end{array}}

\newcommand{\fonctionsansdef}[3]{#1 : #2 \longrightarrow #3}

\begin{document}

\title{Helmholtz theorem for Hamiltonian systems on time scales}
\author{Fr\'ed\'eric Pierret}
\address{SYRTE, Observatoire de Paris, 77 avenue Denfert-Rochereau, 75014 Paris, France}


\maketitle

\begin{abstract}
We derive the Helmholtz theorem for Hamiltonian systems defined on time scales in the context of nonshifted calculus of variations which encompass the discrete and continuous case. Precisely, we give a theorem characterizing first order equation on time scales, admitting a Hamiltonian formulation which is defined with non-shifted calculus of variation. Moreover, in the affirmative case, we give the associated Hamiltonian.
\end{abstract}


\section{Introduction}

A classical problem in Analysis is the well-known {\it Helmholtz's inverse problem of the calculus of variations}: find a necessary and sufficient condition under which a (system of) differential equation(s) can be written as an Euler--Lagrangeagrange or a Hamiltonian equation and, in the affirmative case, find all the possible Lagrangian or Hamiltonian formulations. This condition is usually called \textit{Helmholtz condition}. \\

The Lagrangian Helmholtz problem has been studied and solved by J. Douglas \cite{doug}, A. Mayer \cite{maye} and A. Hirsch \cite{hirs,hirs2}. The Hamiltonian Helmholtz problem has been studied and solved up to our knowledge by R. Santilli in his book \cite{santilli}.\\

Generalization of this problem in the discrete calculus of variations framework has been done in \cite{bourdin-cresson} and \cite{hydon-mansfield} in the discrete Lagrangian case. For the Hamiltonian case it has been done for the discrete calculus of variations in \cite{albu} using the framework of \cite{mars} and in \cite{cresson-pierret} using a discrete embedding procedure. \\

In this paper we generalize the Helmholtz theorem for Hamiltonian systems in the case of time scale calculus using the work of \cite{ctm} and \cite{bourdin} in the context of nonshifted calculus. We recover then the discrete case \cite{cresson-pierret} when the time scale is a set of discrete points and the classical case \cite{santilli} when the time scale is a continuous interval. \\

The paper is organized as follow : In section 2, we recall some results on time scale calculus. In section 3, we give the definition of time scale embedding. In section 4, we give a brief survey of the classical Helmholtz Hamiltonian problem. In section 5, we introduce the definition of a time scale Hamiltonian in the context of nonshifted calculus of variations and then we enunciate and prove the main result of this paper, the time scale Hamiltonian Helmholtz theorem. Finally, in section 6, we conclude and give some prospects.


\section{Reminder about the time scale calculus}
We refer to \cite{agar2,bohn,bohn3} and references therein for more details on time scale calculus. \\

We consider $\T$ which denotes a bounded time scale with $a = \min (\T)$, $b = \max (\T)$ and $\mathrm{card} (\T) \geq 3$.

\begin{definition}
The backward and forward jump operators $\rho, \sigma : \T \longrightarrow \T$ are respectively defined by:
\begin{equation*}
\forall t \in \T, \; \rho (t) = \sup \{ s \in \T, \; s < t \} \; \text{and} \; \sigma (t) = \inf \{ s \in \T, \; s > t \},
\end{equation*}
where we put $\sup \emptyset = a$ and $\inf \emptyset = b$.
\end{definition}

\begin{definition}
A point $t \in \T$ is said to be left-dense (resp. left-scattered, right-dense and right-scattered) if $\rho (t) = t$ (resp. $\rho (t) < t$, $\sigma (t) = t$ and $\sigma (t) > t$).
\end{definition}

Let $\LD$ (resp. $\LS$, $\RD$ and $\RS$) denote the set of all left-dense (resp. left-scattered, right-dense and right-scattered) points of $\T$.\\

\begin{definition}
The graininess (resp. backward graininess) function $\fonctionsansdef{\mu}{\T}{\R^+}$ (resp. $\fonctionsansdef{\nu}{\T}{\R^+}$) is defined by $\mu(t) = \sigma (t) -t$ (resp. $\nu(t) = t- \rho (t)$) for any $t \in \T$.
\end{definition}

We set $\TK = \T \backslash ]\rho(b),b]$, $\Tk = \T \backslash [a,\sigma(a)[$ and $\TKk = \TK \cap \Tk$.

\begin{remark}
Note that $\TKk \neq \emptyset$ since $\mathrm{card} (\T) \geq 3$. 
\end{remark}

Let us recall the usual definitions of $\DD$- and $\nabla$-differentiability.

\begin{definition}
A function $\fonctionsansdef{u}{\T}{\R^n}$, where $n \in \N^*$, is said to be $\DD$-differentiable at $t \in \TK$ (resp. $\nabla$-differentiable at $t \in \Tk$) if the following limit exists in $\R^n$:
\begin{equation}
\lim\limits_{\substack{s \to t \\ s \neq \sigma (t) }} \dfrac{u(\sigma(t))-u(s)}{\sigma(t) -s} \; \left( \text{resp.} \; \lim\limits_{\substack{s \to t \\ s \neq \rho (t) }} \dfrac{u(s)-u(\rho (t))}{s-\rho(t)} \right).
\end{equation}
In such a case, this limit is denoted by $u^\DD (t)$ (resp. $u^\nabla (t)$).
\end{definition}

Let us recall the following results on $\DD$-differentiability.

\begin{theorem}[{\cite[Theorem 1.16 p.5]{bohn}}]
\label{rappeldelta1}
Let $\fonctionsansdef{u}{\T}{\R^n}$ and $t \in \TK$. The following properties hold:
\begin{enumerate}
\item if $u$ is $\DD$-differentiable at $t$, then $u$ is continuous at $t$.
\item if $t \in \RS$ and if $u$ is continuous at $t$, then $u$ is $\DD$-differentiable at $t$ with:
\begin{equation}
u^\DD (t) = \dfrac{u(\sigma(t))-u(t)}{\mu(t)}.
\end{equation}
\item if $t \in \RD$, then $u$ is $\DD$-differentiable at $t$ if and only if the following limit exists in $\R^n$:
\begin{equation}
\lim\limits_{\substack{s \to t \\ s \neq t}} \dfrac{u(t)-u(s)}{t-s}.
\end{equation}
In such a case, this limit is equal to $u^\DD (t)$.
\end{enumerate}
\end{theorem}

\begin{proposition}[{\cite[Corollary 1.68 p.25]{bohn}}]
\label{rappeldelta2}
Let $\fonctionsansdef{u}{\T}{\R^n}$. Then, $u$ is $\DD$-differentiable on $\TK$ with $u^\DD = 0$ if and only if there exists $c \in \R^n$ such that $u(t) =c$ for every $t \in \T$.
\end{proposition}

The analogous results for $\nabla$-differentiability are also valid.

\begin{definition}
A function $u$ is said to be rd-continuous (resp. ld-continuous) on $\T$ if it is continuous at every $t \in \RD$ (resp. $t \in \LD$) and if it admits a left-sided (resp. righ-sided) limit at every $t \in \LD$ (resp. $t \in \RD$).
\end{definition}

We respectively denote by $\Crd(\T)$ and $\Cdrd(\T)$ the functional spaces of rd-continuous functions on $\T$ and of $\DD$-differentiable functions on $\TK$ with rd-continuous $\DD$-derivative. We also respectively denote by $\Cld(\T)$ and $\Cnld(\T)$ the functional spaces of ld-continuous functions on $\T$ and of $\nabla$-differentiable functions on $\Tk$ with ld-continuous $\nabla$-derivative. \\

Let us recall the following results on the continuity of the forward and backward jump.

\begin{proposition}[{\cite[Lemma 1 p.548]{bourdin}}]
\label{propcontinuity}
Let $t \in \Tk$. The following properties are equivalent:
\begin{enumerate}
\item $\sigma$ is continuous at $t$;
\item $\sigma \circ \rho (t) = t$;
\item $t \notin \RS \cap \LD$.
\end{enumerate}

Also, let $t \in \TK$. if the following properties are equivalent:
\begin{enumerate}
\item $\rho$ is continuous at $t$;
\item $\rho \circ \sigma (t) = t$;
\item $t \notin \LS \cap \RD$.
\end{enumerate}
\end{proposition}

\begin{proposition}[{\cite[Corollary 1 p.552]{bourdin}}]
\label{propbourdin}
Let $u : \T \longrightarrow \R^n$. If the following properties are satisfied:
\begin{itemize}
\item $\sigma$ is $\nabla$-differentiable on $\Tk$;
\item $u$ is $\DD$-differentiable on $\TK$;
\end{itemize}
then, $u\circ\sigma$ is $\nabla$-differentiable at every $t \in \TKk$ with

\begin{equation}
(u\circ\sigma)^\nabla (t) = \sigma^\nabla (t) u^\DD (t). \label{bourdinform1}
\end{equation}

Also, if the following properties are satisfied:
\begin{itemize}
\item $\rho$ is $\DD$-differentiable on $\TK$;
\item $u$ is $\nabla$-differentiable on $\Tk$;
\end{itemize}
then, $u\circ\rho$ is $\DD$-differentiable at every $t \in \TKk$ with

\begin{equation}
(u\circ\rho)^\DD (t) = \rho^\DD(t) u^\nabla (t). \label{bourdinform2}
\end{equation}

\end{proposition}

Let us now derive a result about $\DD$-differentiability of $\rho$ and the $\NN$-differentiability of $\sigma$.

\begin{proposition}\label{propinv}
If $\sigma$ is $\NN$-differentiable on $\Tk$ and $\rho$ is $\DD$-differentiable on $\TK$, then we have $\rho^\DD(t) \sigma^\NN(t)=1$ for all $t\in \TKk$.
\end{proposition}

\begin{proof}
If $\sigma$ is $\NN$-differentiable on $\Tk$ and $\rho$ is $\DD$-differentiable on $\TK$, then the following limits then exist and are finite

\begin{equation*}
\lim\limits_{\substack{s \to t \\ s \neq \sigma (t) }} \dfrac{\rho(\sigma(t))-\rho(s)}{\sigma(t) -s} \quad \text{and} \quad  \lim\limits_{\substack{s \to t \\ s \neq \rho (t) }} \dfrac{\sigma(s)-\sigma(\rho (t))}{s-\rho(t)}.
\end{equation*}
and also $\sigma$ is continuous on $\Tk$ and $\rho$ is continuous on $\TK$ (Proposition \eqref{rappeldelta1}). By the continuity of $\sigma$ and $\rho$ (Proposition \eqref{propcontinuity}) we have only only two cases to consider, $t\in \LS\cap\RS$ or $t\in \LD\cap\RD$. If $t\in \LS\cap\RS$, then $\sigma^\NN(t)=\frac{\mu(t)}{\nu(t)}$ and $\rho^\DD(t)=\frac{\nu(t)}{\mu(t)}$ and then we have $\rho^\DD(t) \sigma^\NN(t)=1$. If $t\in \LD\cap\RD$, then $\sigma^\NN(t)=1$ and $\rho^\DD(t)=1$ and then we have also the result.

\end{proof}

Let us denote by $\int \DD \tau$ the Cauchy $\DD$-integral defined in \cite[p.26]{bohn} with the following result.

\begin{theorem}[{\cite[Theorem 1.74 p.27]{bohn}}]
For every $u \in \Crd(\TK)$, there exist a unique $\DD$-antiderivative of $u$ in sense of $U^\DD = u$ on $\TK$ vanishing at $t=a$. In this case the $\DD$-integral is defined by
\begin{equation*}
U(t) = \int_a^t u(\tau) \DD \tau
\end{equation*}
for every $t \in \T$.
\end{theorem}

\begin{proposition}
The $\DD$-integral has the following properties:

\begin{itemize}
	
\item[(i)]if $c,d\in\mathbb{T}$ and $f$ and $g$
are $\DD$-differentiable, then the following formula
of integration by parts hold:
\begin{equation*}
\label{ippdelta1}
\int_{c}^{d} f(t) g^\DD(t)\Delta t
=\left.(fg)(t)\right|_{t=c}^{t=d}
-\int_{c}^{d} f^\DD(t) g(\sigma(t)) \Delta t.
\end{equation*}
	
\item[(ii)]if $c,d\in\mathbb{T}$, $f$ is $\nabla$-differentiable and $g$ and $\rho$ are $\DD$-differentiable, then the following formula
of integration by parts hold:
\begin{equation*}
\label{ippdelta2}
\int_{c}^{d} f(t) g^\DD(t)\Delta t
=\left.f(\rho(t))g(t)\right|_{t=c}^{t=d}
-\int_{c}^{d} \rho^\DD(t) f^\nabla(t) g(t) \Delta t.
\end{equation*}

\end{itemize}
\end{proposition}

The first integration by parts formula (i) is well known in the litterature of the time scale calculus. The second (ii) is obtain by $\DD$-integration of the Leibniz formula \cite[Proposition 7 p.552]{bourdin}. \\

Let us now remind some definitions of variationnal calculus.

\begin{definition}
Let $L$ be a Lagrangian \textit{i.e.} a continuous map of class $C^1$ in its two last variables:
\begin{equation*}
\fonction{L}{\TK \times \R^n \times \R^n}{\R}{(t,x,v)}{L(t,x,v)}
\end{equation*}
and let $\L$ be the following Lagrangian functional:
\begin{equation*}
\fonction{\L}{\Cdrd(\T)}{\R}{u}{\di \int_a^b L(\tau,u(\tau),u^\DD(\tau)) \DD \tau.}
\end{equation*}
We define $\Cdrdz(\T) = \{ w \in \Cdrd (\T), \; w(a)= w(b) = 0\}$ to be the set of \textit{variations} of $\L$. A function $u \in \Cdrd(\T)$ is said to be a \textit{critical point} of $\L$ if $D\L(u)(w) = 0$ for every $w \in \Cdrdz(\T)$ where $D$ denotes the Frechet derivative.
\end{definition}

We remind what we call a strong form of the well known Dubois--Reymond lemma on time scales.

\begin{lemma}[Dubois--Reymond strong form {\cite[Lemma 4.1]{bohn2}}]
Let $q\in \mathrm{C}^{0}_{\mathrm{rd}}(\TK,\R^n)$. Then the equality

\begin{equation*}
\int_{a}^{b} q(\tau) \cdot w^\DD(\tau) \DD \tau = 0
\end{equation*}

holds for every $w\in \Cdrdz(\T,\R^n)$ if and only if there exist $c\in\R^n$ such that $q(t)=c$ for all $t\in \TK$.

\end{lemma}

Up to our knowledge, the classical form of the Dubois--Reymond lemma on time scales has not been used or proved. We give the result and we call it the weak form of the Dubois--Reymond lemma.

\begin{lemma}[Dubois--Reymond weak form]
Let $q\in \Crd(\TK,\R^n)$. Then the equality

\begin{equation*}
\int_{a}^{b} q(\tau) \cdot w(\tau) \DD \tau = 0
\end{equation*}

holds for every $w\in \Crd(\TK,\R^n)$ such that $w(a)=w(b)=0$ if and only if $q(t)=0$ for all $t\in \TKk$.

\end{lemma}

\begin{proof}
The sufficient condition is obvious. For the necessary one : Let $r(\tau)=(\tau-a)^2(\tau-b)^2$ which is clearly positive for all $\tau \in \T$ and vanish at $\tau=a$ and $\tau=b$. Let $w(\tau)=r(\tau)q(\tau)$. We have
\begin{equation*}
0=\int_{a}^{b} q(\tau) \cdot w(\tau) \DD \tau = \int_{a}^{b} \|q(\tau)\|^2 r(\tau)\DD \tau.
\end{equation*}

As $\|q(\tau)\|^2 r(\tau)\ge0$ for all $t\in \TKk$ then necessarily we have $q(\tau)=0$ for all $t\in \TKk$.

\end{proof}


\section{Time scale embeddings}
We remind the time scale embedding as defined in \cite{ctm} to which we refer for more details.\\

We denote by $C([a,b];\R)$ the set of continuous functions $x : [a,b] \rightarrow \R$. A time scale embedding is given by specifying:
\begin{itemize}
\item A mapping $\iota : C([a,b],\R) \rightarrow C(\T,\R )$;
\item An operator $\delta : C^1([a,b],\R) \rightarrow \Cdrd(\TK,\R)$,
called a generalized derivative;
\item An operator $J: C([a,b],\R)\rightarrow \Crd(\T,\R)$,
called a generalized integral operator.
\end{itemize}

We fix the following embedding:

\begin{definition}[Time scale $\DD$-embedding]
The mapping $\iota$ is obtained by restriction of functions to $\T$.
The operator $\delta$ is chosen to be the $\Delta$ derivative, and the operator
$J$ is given by the $\Delta$-integral as follows:
$$
\delta u(t) := u^\DD(t) \, ,  \quad
Ju(t) := \di\int_a^{\sigma (t)} u(s) \Delta s\, .
$$
\end{definition}

\begin{definition}[Time scale $\Delta$-embedding of differential equations]
\label{embeddingdofdifferential}
The $\DD$-differential embedding of an ordinary differential equation

\begin{equation*}
\frac{dx(t)}{dt}=f(t,x(t))
\end{equation*}
for $x \in C^1([a,b],\R)$ and $f\in C(\R \times C^1([a,b],\R), \R)$, is given by

\begin{equation*}
x^\DD(t)=f(t,x(t))
\end{equation*}
for $x\in \Cdrd(\TK,\R)$ and $f\in C(\T \times \Cdrd(\TK,\R), \R)$.
\end{definition}

\begin{definition}[Time scale $\Delta$-embedding of integral equations]
\label{embeddingdofintegral}
The $\DD$-integral embedding of an integral equation

\begin{equation*}
x(t)=x(a)+\int_a^t f(s,x(s))ds
\end{equation*}
for $x \in C^1([a,b],\R)$ and $f\in C(\R \times C^1([a,b],\R), \R)$, is given by

\begin{equation*}
x(t)=x(a)+\int_a^{\sigma(t)} f(s,x(s))\DD s
\end{equation*}
for $x\in \Cdrd(\TK,\R)$ and $f\in C(\T \times \Cdrd(\TK,\R), \R)$.
\end{definition}

\begin{definition}[Time scale $\DD$-embedding of integral functionals]
\label{embeddingfunctional}
Let $L : [a,b] \times \R^2 \rightarrow \R$ be a continuous function and $\mathcal{L}$
the functional defined by
\begin{equation*}
\mathcal{L}(x) = \int_a^t L\left(s,x(s),\frac{dx(s)}{dt}\right) ds.
\end{equation*}
The time scale $\DD$-embedding $\mathcal{L}_{\Delta}$ of $\mathcal{L}$ is given by
\begin{equation*}
\mathcal{L}_{\Delta}(x) =\int_a^{\sigma(t)} L\left(s,x(s),x^\DD(s)\right) \Delta s.
\end{equation*}
\end{definition}


\section{Classical Helmholtz Hamiltonian problem : a brief survey}
This section is based on the book of \cite{santilli} to which we refer for more details. \\

\subsection{Generalities and notations}

We work on $\R^{2d},d \ge 1, d \in \N$. We denote by ${}^T$ the transpose. The symplectic scalar product ${< \cdot,\cdot>}_J$ is defined for all $X,Y \in \R^{2d}$ by

\begin{equation*}
{<X,Y>}_J = <X,J\cdot Y>
\end{equation*}

where $< \cdot,\cdot>$ denote the usual scalar product and $J = \begin{pmatrix} 0 & I_d \\ -I_d & 0 \end{pmatrix}$ with $I_d$ the identity matrix on $\R^d$.

\begin{definition}
We define the $L^2$ symplectic scalar product induced by ${<\cdot,\cdot>_J}$ defined for $f,g \in C^1([a,b],\R^{2d})$ by 

\begin{equation*}
<f,g>_{L^2,J}\ =\int_{a}^{b}<f(t),g(t)>_Jdt \ .
\end{equation*}

\end{definition}

\begin{definition}
Let $\fonctionsansdef{A}{C^0([a,b],\R^{2d})}{C^0([a,b],\R^{2d})}$. We define the adjoint $A^*_J$ of $A$ with respect to $<\cdot,\cdot>_{L^2,J}$ by

\begin{equation*}
<A \cdot f, g>_{L^2,J} = <A^{*}_J \cdot g, f>_{L^2,J} \ .
\end{equation*}
\end{definition}

\begin{definition}[Classical Hamiltonian]
A classical Hamiltonian is a function $H : \R^d \times \R^d \rightarrow \R$ such that for $(q,p)\in C^1([a,b],\R^d) \times C^1([a,b],\R^d)$ we have the time evolution of $(q,p)$ given by the classical Hamilton's equations

\begin{align*}
\left\{
\begin{array}{l l}
\frac{dq}{dt}&=\frac{\partial H(q,p)}{\partial p} \\
\frac{dp}{dt}&=-\frac{\partial H(q,p)}{\partial q}
\end{array}
\right.
\label{def_hamilton}
\end{align*}

\end{definition}

\begin{remark}
We say that $X=\begin{pmatrix} \frac{\partial H(q,p)}{\partial p} \\ -\frac{\partial H(q,p)}{\partial q} \end{pmatrix}$ is Hamiltonian.
\end{remark}

\begin{theorem}
The critical points $(q,p)\in C^1([a,b],\R^d) \times C^1([a,b],\R^d)$ of the functional 

\begin{equation*}
\fonction{\mathcal{L}_H}{C^1([a,b],\R^d)\times C^1([a,b],\R^d)}{\R}{(q,p)}{\mathcal{L}_H(q,p) = \int_{a}^{b}L_H(q(t),p(t),\dot{q}(t),\dot{p}(t))dt}
\end{equation*}

satisfy the Hamilton's equations where $\fonctionsansdef{L_H}{\R^d \times \R^d \times \R^d \times \R^d}{\R}$ is the Lagrangian defined by

\begin{equation*}
L_H(x,y,v,w)=<y,v>-H(x,y)
\end{equation*}

\label{thm_hamiltonian}
\end{theorem}

\subsection{Classical Helmholtz Hamiltonian theorem}

We consider the differential equations associate to a vector field $X=\begin{pmatrix} X_q \\ X_p \end{pmatrix}$ ,

\begin{align*}
\frac{d}{dt}\begin{pmatrix}q \\ p \end{pmatrix} = \begin{pmatrix} X_q(q,p) \\ X_p(q,p) \end{pmatrix}.
\end{align*}

It defines a natural operator which is written as

\begin{equation*}
O_X(q,p) = \begin{pmatrix} \frac{dq}{dt} - X_q(q,p) \\ \frac{dp}{dt} - X_p(q,p) \end{pmatrix} \ .
\end{equation*}

In this case we have the Hamiltonian Helmholtz conditions :

\begin{proposition}[Classical Hamiltonian Helmholtz conditions {\cite[Theorem 2.7.3, p.88]{santilli}}]
The operator $O_X$ has its Fr\'echet derivative self-adjoint at $(q,p)$ if and only if
\begin{align*}
& \frac{\partial X_q(q,p)}{\partial q} + \left(\frac{\partial X_p(q,p)}{\partial p} \right)^T = 0 \\
& \frac{\partial X_q(q,p)}{\partial p} \ \text{and} \ \frac{\partial X_p(q,p)}{\partial q} \ \text{are symmetric}.
\end{align*}

\end{proposition}

We have then

\begin{theorem}[{\cite[Theorem 3.12.1-2, p.176-177]{santilli}}]
The vector field X is Hamiltonian if and only if the operator $O_X$ has his Fr\'echet derivative self-adjoint with respect to the symplectic scalar product. \\
	
In this case the Hamiltonian associate to $X$ is given by 
	
\begin{equation*}
H(q,p)=\int_{0}^{1}\left[p \cdot X_q(\lambda q, \lambda p) - q\cdot X_p(\lambda q, \lambda p) \right]d\lambda.
\end{equation*}
	
\end{theorem} 

\begin{remark}
The Classical Hamiltonian Helmholtz conditions are also the conditions to which the differential form associate to the vector field $X$ with respect to the symplectic scalar product is closed.
\end{remark}

\section{Time scale Helmholtz Hamiltonian problem}

\subsection{Generalities and notations}

\begin{definition}
We define the $L^2-\DD$ scalar product defined for $f,g \in \Crd(\T,\R^{d})$ by 

\begin{equation*}
<f,g>_{L^2,\DD}=\int_{a}^{b}<f(t),g(t)>\DD t \ .
\end{equation*}

and also the $L^2-\DD$ symplectic scalar product defined for $f,g \in C^0_{rd}(\T,\R^{2d})$ by 

\begin{equation*}
<f,g>_{L^2,\DD,J}=<f,J\cdot g>_{L^2,\DD}
\end{equation*}
\end{definition}

\begin{definition}
Let $\fonctionsansdef{A}{\Crd(\T,\R^{2d})}{\Crd(\T,\R^{2d})}$. We define the adjoint $A^*_J$ of $A$ with respect to $<\cdot,\cdot>_{L^2,\DD,J}$ by

\begin{equation*}
<A \cdot f, g>_{L^2,\DD,J} = <A^{*}_J \cdot g, f>_{L^2,\DD,J} \ .
\end{equation*}
\end{definition}

Let $\CDNI(\T)$ denote the set $\CDN$ and $\CDNz(\T)=\{ w\in \CDNI(\T) , w(a)=w(b)=0 \}$.

\subsection{Time scale Hamiltonian}

We consider the $\DD$-embedding of $\mathcal{L}_H$, $\mathcal{L}_{H,\DD}$ defined by

\begin{equation*}
\mathcal{L}_{H,\DD}(q,p) = \int_{a}^{b}L_H(q(t),p(t),q^\DD(t),p^\DD(t)) \Delta t
\end{equation*}

for all $(q,p)\in \CDNI(\T,\R^d)$. \\

We assume that $\rho$ is $\DD$-differentiable on $\TK$ and $\sigma$ is $\nabla$-differentiable on $\Tk$. \\

\begin{definition}[Time scale Hamiltonian]
A time scale Hamiltonian is a function $H : \R^d \times \R^d \rightarrow \R$ such that for $(q,p)\in \CDNI(\T,\R^d)$ we have the time evolution of $(q,p)$ given by the time scale Hamilton's equations under the \textit{derivative form}

\begin{align*}
(\star 1) \left\{
\begin{array}{r l}
q^\DD(t)&=\frac{\partial H(q(t),p(t)}{\partial p} \\
\rho^\DD(t)p^\nabla(t)&=-\frac{\partial H(q(t),p(t))}{\partial q}
\end{array}
\right. \quad \text{for all} \ t\in \TKk
\end{align*}

or under the \textit{integral form}

\begin{align*}
(\star 2) \left\{
\begin{array}{r l}
q(\sigma(t))&=\di \int_{a}^{\sigma(t)} \frac{\partial H}{\partial p}(q(\tau),p(\tau))\DD\tau + C_q \\
p(t)&=\di \int_{a}^{\sigma(t)}-\frac{\partial H}{\partial q}(q(\tau),p(\tau))\DD\tau + C_p
\end{array}
\right. \quad \text{for all} \ t\in \TK,
\end{align*}

 where $C_q$ and $C_p$ are constants. Moreover, the derivative form and the integral form are equivalent.

\end{definition}

\begin{theorem}
The critical points $(q,p)\in \CDNI(\T,\R^d)$ of the functional $\mathcal{L}_{H,\DD}$ satisfy the time scale Hamilton's equations $(\star 1)$ or equivalently $(\star2)$.
\end{theorem}

\begin{proof}

Let $(u,v) \in \CDNz(\T,\R^d)$. The Fr\'echet derivative of $\mathcal{L}_{H,\DD}$ at $(q,p)$ along $(u,v)$ is given by
\begin{equation*}
D\mathcal{L}_{H,\Delta}(q,p)(u,v)=\di \int_{a}^{b} \left[ p(t) \cdot u^\DD(t)+v(t)\cdot q^\DD(t)- DH\left(q(t),p(t)\right)(u(t),v(t))\right]\DD t.
\end{equation*}
where the Fr\'echet derivative of $H$ at $(q,p)$ along $(u,v)$ is given by

\begin{equation*}
DH(q,p)(u,v)=\frac{\partial H(q,p)}{\partial q}\cdot u + \frac{\partial H(q,p)}{\partial p}\cdot v \ .
\end{equation*}
First we prove the critical points satisfy the Hamilton's equations under the derivative form.\\

Using the integration by parts formula $(ii)$ of Proposition \eqref{int:par:delta2} and using the fact that $u$ vanish at $t=a$ and $t=b$ we obtain

\begin{equation*}
D\mathcal{L}_{H,\Delta}(q,p)(u,v)=\di \int_{a}^{b} \left[ -\rho^\DD(t) \NN p(t) \cdot u(t)+v(t)\cdot q^\DD(t)- DH\left(q(t),p(t)\right)(u(t),v(t))\right]\DD t.
\end{equation*}

Using the expression $DH(q,p)(u,v)$ we obtain

\begin{align*}
& D\mathcal{L}_{H,\Delta}(q,p)(u,v) \\
&=\di \int_{a}^{b} \left[ -u(t)\cdot \left(\rho^\DD(t) \NN p(t)+\frac{\partial H(q(t),p(t))}{\partial q}\right)+v(t)\cdot\left(q^\DD(t)-\frac{\partial H(q(t),p(t)}{\partial p}\right)\right]\DD t.
\end{align*}

By definition, if $(q,p)$ is a critical point of $D\mathcal{L}_{H,\Delta}$ then we have

\begin{equation*}
D\mathcal{L}_{H,\Delta}(q,p)(u,v)=0
\end{equation*}

and then using the weak form of the Dubois--Reymond Lemma we obtain the time scale Hamilton's equations under the derivative form. \\

Second using the same strategy with the integration by parts formula $(i)$ of Proposition \eqref{int:par:delta1} and using the strong form of the Dubois--Reymond Lemma we obtain the time scale Hamilton's equations under the integral form. \\

Equivalence between $(\star 1)$ and $(\star 2)$ is due to the $\DD$-differentiability of $\rho$ on $\TK$ and the $\nabla$-differentiability of $\sigma$ on $\Tk$. Indeed, using Proposition \eqref{propinv} we obtain for all $t\in \TKk$

\begin{equation*}
(\star 1) \Longleftrightarrow 
\left\{
\begin{array}{r l}
\sigma^\NN(t) q^\DD(t)&=\di\sigma^\NN(t)\frac{\partial H(q(t),p(t))}{\partial p} \\
p^\nabla(t)&=\di-\sigma^\NN(t)\frac{\partial H(q(t),p(t))}{\partial q}
\end{array}
\right. \ .
\end{equation*}

Then using Proposition \eqref{propbourdin} we obtain for all $t\in \TKk$

\begin{align*}
\left\{
\begin{array}{r l}
\sigma^\NN(t) q^\DD(t)&=\di\sigma^\NN(t)\frac{\partial H(q(t),p(t))}{\partial p} \\
p^\nabla(t)&=\di-\sigma^\NN(t)\frac{\partial H(q(t),p(t))}{\partial q}
\end{array}
\right. \ \Longleftrightarrow \
\left\{
\begin{array}{r l}
\left[q(\sigma(t))\right]^\NN&=\di \left[\int_{a}^{\sigma(t)}\frac{\partial H(q(\tau),p(\tau))}{\partial p}\DD\tau\right]^\NN \\
p^\nabla(t)&=\di \left[\int_{a}^{\sigma(t)} -\frac{\partial H(q(\tau),p(\tau))}{\partial q}\DD\tau\right]^\NN
\end{array}
\right. \ .
\end{align*}

Using the $\NN$ version of Proposition \eqref{rappeldelta2} we obtain $(\star 1) \Longleftrightarrow (\star 2)$.

\end{proof}

\subsection{Time scale Helmholtz Hamiltonian Theorem}

We consider the general system of time scale equations associate to $X=\begin{pmatrix} X_q \\ X_p \end{pmatrix}$ defined as

\begin{align*}
\left(\star\right) \ \left\{
\begin{array}{r l}
q^\DD(t)&=X_q(q(t),p(t)) \\
\rho^\DD(t)p^\nabla(t)&=X_p(q(t),p(t))
\end{array}
\right.
\end{align*}

for all $t\in \TKk$. It defines an operator $O^\T_X$ as follows :

\begin{equation*}
\fonction{O^\T_X}{\CDNI(\T,\R^{2d})}{C(\TKk,\R^{2d})}{(q,p)}{\begin{pmatrix} q^\DD - X_q(q,p) \\ \rho^\DD p^\nabla - X_p(q,p) \end{pmatrix} \ .}
\end{equation*}
A straightforward computation leads to :

\begin{proposition}

Let $(u,v) \in \CDNz(\T,\R^d)$. The Fr\'echet derivative $DO^\T_X(q,p)$ is given by

\begin{equation*}
DO^\T_X(q,p)(u,v) = \begin{pmatrix} u^\DD -\frac{\partial X_q}{\partial q} \cdot u -\frac{\partial X_q }{\partial p}\cdot v \\ \rho^\DD v^\nabla -\frac{\partial X_p}{\partial q} \cdot u -\frac{\partial X_p}{\partial p} \cdot v \end{pmatrix}
\end{equation*}

and its adjoint ${DO^{\T,*}_{X,J}}(q,p)$ with respect to the $L^2-\DD$ symplectic scalar product is given by

\begin{equation*}
{DO^{\T,*}_{X,J}}(q,p)(u,v) = \begin{pmatrix} u^\DD +\left(\frac{\partial X_p}{\partial p}\right)^T \cdot u-\left(\frac{\partial X_q}{\partial p}\right)^T\cdot v \\ \rho^\DD v^\nabla -\left(\frac{\partial X_p}{\partial q}\right)^T \cdot u +\left(\frac{\partial X_q}{\partial q}\right)^T \cdot v \end{pmatrix}.
\end{equation*}

\end{proposition}
By identification we obtain :

\begin{lemma}[Time scale Hamiltonian Helmholtz conditions]
The operator $O^\T_X$ has its Fr\'echet derivative self-adjoint at $(q,p)\in \CDNI(\T,\R^d)$ if and only if the following conditions are satisfied over $\TKk$ :

\begin{align*}
& \frac{\partial X_q(q,p)}{\partial q} + \left(\frac{\partial X_p(q,p)}{\partial p} \right)^T = 0 \\
& \frac{\partial X_q(q,p)}{\partial p} \ \text{and} \ \frac{\partial X_p(q,p)}{\partial q} \ \text{are symmetric}.
\end{align*}

\end{lemma}

Now we can state the main result of this paper :

\begin{theorem}[Helmholtz theorem for Hamiltonian systems on time scales]
\label{helmholtztimescale}
The vector field X is a time scale Hamiltonian if and only if the operator $O^\T_X$ associated has his Fr\'echet derivative self-adjoint with respect to the $L^2-\DD$ symplectic scalar product. \\
	
In this case the Hamiltonian associate to $X$ is given by 
\begin{equation*}
H(q,p)=\int_{0}^{1}\left[p \cdot X_q(\lambda q, \lambda p) - q\cdot X_p(\lambda q, \lambda p) \right]d\lambda
\end{equation*}
	
\end{theorem} 

\begin{proof}

Let $(u,v) \in \CDNz(\T,\R^d)$. The Fr\'echet derivative of $\mathcal{L}_{H,\DD}$ at $(q,p)$ along $(u,v)$ is given by
\begin{equation*}
D\mathcal{L}_{H,\Delta}(q,p)(u,v)=\di \int_{a}^{b} \left[ p(t) \cdot u^\DD(t)+v(t)\cdot q^\DD(t)- DH\left(q(t),p(t)\right)(u(t),v(t))\right]\DD t.
\end{equation*}
The time scale Hamiltonian Helmholtz conditions implies that the Fr\'echet derivative of $H$ at $(q,p)$ along $(u,v)$ is given by

\begin{equation*}
DH(q,p)(u,v)=\int_{0}^{1} \frac{\partial}{\partial \lambda} \left(v\cdot X_q(q,p)-u\cdot X_p(q,p)\right) d\lambda,
\end{equation*}
which leads to
\begin{equation*}
DH(q,p)(u,v)=v\cdot X_q(q,p)-u\cdot X_p(q,p).
\end{equation*}

By definition, if $(q,p)$ is a critical point of $D\mathcal{L}_{H,\Delta}$ then we have
\begin{equation*}
D\mathcal{L}_{H,\Delta}(q,p)(u,v)=0
\end{equation*}
and then using integration by parts formula $(ii)$ with the weak form of the Dubois--Reymond lemma concludes the proof.
\end{proof}

\section{Conclusion and prospects}

We proved a result on first order time scale equations which allows us to find the existence of a Hamiltonian structure associate and in the affirmative case to give the Hamiltonian. Our result recover both the discrete and classical case but it allows the mixing of both of them as the time scale calculus was created with such motivation \cite{hilger}. The Hamiltonian Helmholtz problem was easier to prove as contrary to the Lagrangian case. Indeed, there is no mixing of $\DD$ and $\NN$ derivative such as $\DD \circ \NN$ or $\NN \circ \DD$ which happen in the discrete case \cite{bourdin-cresson}. \\

The further extension of the Hamiltonian Helmholtz problem is to consider the derivative as combinations of $\DD$ and $\NN$ such that $\diamond=\frac{\DD+\NN}{2}$ which is the diamond integral for which motivations and definitions can be found in \cite{rogers}, \cite{dacruz} and references therein. \\

Another further extension of this result concern the stochastic calculus and more precisely the stochastic calculus on time scales defined in \cite{sanyal} and \cite{bohnersto}. The work is to define a natural notion of stochastic Hamiltonian on time scales and then to give the stochastic version of Theorem \ref{helmholtztimescale}. This extension is a work in progress and will be the subject of a future paper.

\bibliographystyle{plain}


\end{document}